\documentclass[12pt,fleqn,a4paper]{amsart}

\usepackage[utf8]{inputenc}
\usepackage[foot]{amsaddr}
\usepackage[pdftex]{lscape}
\usepackage{comment}
\specialcomment{scomment}{\begingroup\color{red}}{\endgroup}

\usepackage[in]{fullpage}
\usepackage[onehalfspacing]{setspace}

\usepackage{amsfonts,amssymb,amsthm}

\usepackage[pagebackref,colorlinks,breaklinks,linkcolor=blue,anchorcolor=blue,citecolor=blue]{hyperref}

\usepackage[alphabetic,nobysame]{amsrefs}

\catcode`\@=11

\def\@settitle{%
  \baselineskip14\p@\relax
    {\Large\bfseries
  \@title}}

\def\@setauthors{%
  \begingroup
  \def\thanks{\protect\thanks@warning}%
  \trivlist
  \footnotesize \@topsep45\p@\relax
  \advance\@topsep by -\baselineskip
  \item\relax
  \author@andify\authors
  \def\\{\protect\linebreak}%
  {\sc\fontsize{12}{10}\selectfont\authors}%
  \ifx\@empty\contribs
  \else
    ,\penalty-3 \space \@setcontribs
    \@closetoccontribs
  \fi
  \endtrivlist
  \endgroup
}

\def\@secnumfont{\bfseries}%

\def\section{\@startsection{section}{1}%
  \z@{.7\linespacing\@plus\linespacing}{.5\linespacing}%
  {\normalfont\bf}}

\DefineSimpleKey{bib}{arxiv}

\BibSpec{article}{%
    +{}  {\PrintAuthors}                {author}
    +{,} { \textit}                     {title}
    +{.} { }                            {part}
    +{:} { \textit}                     {subtitle}
    +{,} { \PrintContributions}         {contribution}
    +{.} { \PrintPartials}              {partial}
    +{,} { }                            {journal}
    +{}  { \textbf}                     {volume}
    +{}  { \PrintDatePV}                {date}
    +{,} { \eprintpages}                {pages}
    +{,} { }                            {status}
    +{}  { \PrintTranslation}           {translation}
    +{;} { \PrintReprint}               {reprint}
    +{.} { }                            {note}
    +{,} { \arxiv}           {arxiv}
    +{.} {}                             {transition}
  }

\BibSpec{collection.article}{%
    +{}  {\PrintAuthors}                {author}
    +{,} { \textit}                     {title}
    +{.} { }                            {part}
    +{:} { \textit}                     {subtitle}
    +{,} { \PrintContributions}         {contribution}
    +{,} { \PrintConference}            {conference}
    +{}  {\PrintBook}                   {book}
    +{,} { }                            {booktitle}
    +{} { \PrintDate}                  {date}
    +{,} { }                        {pages}
    +{,} { }                            {status}
    +{,} { \PrintDOI}                   {doi}
    +{,} { available at \eprint}        {eprint}
    +{,} { \arxiv}           {arxiv}
    +{}  { \PrintTranslation}           {translation}
    +{;} { \PrintReprint}               {reprint}
    +{.} { }                            {note}
    +{.} {}                             {transition}
}

\BibSpec{book}{%
    +{}  {\PrintPrimary}                {transition}
    +{,} { \textit}                     {title}
    +{.} { }                            {part}
    +{:} { \textit}                     {subtitle}
    +{,} { \PrintEdition}               {edition}
    +{}  { \PrintEditorsB}              {editor}
    +{,} { \PrintTranslatorsC}          {translator}
    +{,} { \PrintContributions}         {contribution}
    +{,} { }                            {series}
    +{,} { \voltext}                    {volume}
    +{,} { }                            {publisher}
    +{,} { }                            {organization}
    +{,} { }                            {address}
    +{} { \PrintDate}                 {date}
    +{,} { }                            {status}
    +{}  { \parenthesize}               {language}
    +{}  { \PrintTranslation}           {translation}
    +{;} { \PrintReprint}               {reprint}
    +{.} { }                            {note}
    +{.} {}                             {transition}
}

\newcommand{\arxiv}[1]{\tt arxiv:\hspace{0pt}{\href{http://arxiv.org/abs/#1}{#1}}}

\renewcommand{\voltext}{\IfEmptyBibField{series}{\textbf}{\textbf}}

\renewcommand{\BibLabel}{%
    \Hy@raisedlink{\hyper@anchorstart{cite.\CurrentBib}\hyper@anchorend}%
    [\thebib]}

  \catcode`\@=12

\numberwithin{equation}{section}\swapnumbers
\usepackage{parskip}

\usepackage[cmtip,matrix,arrow,curve]{xy}
\SelectTips{cm}{10}
\newcommand{\cxymatrix}[1]{\vcenter{\xymatrix@=15pt{#1}}}
\newcommand{\kxymatrix}[1]{\vcenter{\xymatrix@=5pt{#1}}}
\usepackage{rotating}

\renewcommand{\rightleftarrows}[4]%
{\xymatrix@=10pt@M=0pt{#1\,\ar@<2pt>[r]^>>{#2}&\,#4\ar@<2pt>[l]^<<{#3}}}

\newcommand{\xysubseteq}{\ar@{}[r]|{\displaystyle\subseteq}}
\newcommand{\xysubseteqdown}{\ar@{}[d]|{\rotatebox{90}{$\supseteq$}}}

\usepackage{tikz}
\usetikzlibrary{calc}

\usepackage{aliascnt}

\newtheorem{theorem}{Theorem}[section]

\newaliascnt{lemma}{theorem}
\newtheorem{lemma}[lemma]{Lemma}
\aliascntresetthe{lemma}

\newaliascnt{corollary}{theorem}
\newtheorem{corollary}[corollary]{Corollary}
\aliascntresetthe{corollary}

\newaliascnt{proposition}{theorem}
\newtheorem{proposition}[proposition]{Proposition}
\aliascntresetthe{proposition}

\newaliascnt{conjecture}{theorem}

\aliascntresetthe{conjecture}

\theoremstyle{definition}
\newaliascnt{definition}{theorem}

\aliascntresetthe{definition}

\newaliascnt{remark}{theorem}
\newtheorem{remark}[remark]{Remark}
\aliascntresetthe{remark}
\newtheorem{remarks}[remark]{Remarks}

\newtheorem*{remark*}{Remark}

\newaliascnt{example}{theorem}

\aliascntresetthe{example}

\usepackage{enumitem}
\setlist[enumerate,2]{label=\textit{\alph*)},ref=\textit{\alph*}),noitemsep}
\setlist[enumerate,1]{label=\textit{\roman*)},ref=\textit{\roman*}),noitemsep}

\usepackage[capitalize]{cleveref}
\usepackage{microtype}

\renewcommand\[{\begin{equation}}
\renewcommand\]{\end{equation}}

\newcommand\G{{\bf G}}

\renewcommand\phi{\varphi}
\renewcommand\epsilon{\varepsilon}

\renewcommand\theta{\vartheta}
\renewcommand\rho{\varrho}

\newcommand\QQ{{\mathbb Q}}

\newcommand\ZZ{{\mathbb Z}}

\newcommand\cD{{\mathcal D}}

\newcommand\cO{{\mathcal O}}

\newcommand\Bq{{\overline{B}}}

\newcommand\Dq{{\overline{D}}}
\newcommand\eq{{\overline{e}}}

\newcommand\Gq{{\overline{G}}}
\newcommand\Hq{{\overline{H}}}

\renewcommand\qq{{\overline{q}}}

\newcommand\Tq{{\overline{T}}}

\newcommand\xq{{\overline{x}}}\newcommand\Xq{{\overline{X}}}

\newcommand{\into}{\hookrightarrow}

\newcommand\<{\langle}
\renewcommand\>{\rangle}
\def\|#1|{\operatorname{#1}}

\title[]{On the fundamental group of a spherical variety}
\author[]{Friedrich Knop}
\address[]{Dept. Mathematik\\FAU Erlangen-Nürnberg\\
  Cauerstraße 11\\
  D-91058 Erlangen}

\begin{document}

\begin{abstract}
  Let $X=G/H$ be a spherical variety over an algebraically closed
  field of characteristic $p\ge0$. We compute the $p'$-parts
  of $\pi_0(H)$ and $\pi_1(X)$ from the spherical system of $X$.
\end{abstract}

\maketitle         

\section{Introduction}

This paper originated from a question by Spencer Leslie \cite{Leslie}
on MathOverflow on how to compute $\pi_0(H)$ for a spherical subgroup
$H$ of a connected reductive group $G$ from its spherical system. It
turned out that this problem had already been solved by Hofscheier
\cite{Hofscheier}*{Cor.~1.5} where it is a byproduct of much more
general results.  Since the answer to Leslie's question is actually
quite simple, we thought it to be worthwhile to write up a more
straightforward account. For added value, we generalized the result in
two ways. First, we extended it to fields of arbitrary characteristic
$p\ge0$ and thereby compute the $p'$-part of $\pi_0(H)$. Secondly, since
$\pi_0(H)$ is a quotient of the étale fundamental group of $G/H$, we
also compute the entire $p'$-part of $\pi_1(G/H)$.

\section{Generalities}

All varieties will be defined over an algebraically closed field $k$
with characteristic exponent $p\ge1$. A \emph{cover} of a variety $X$
will be a finite étale morphism $\Xq\to X$ such that $\Xq$ is also a
variety, i.e., irreducible. It is a \emph{$p'$-cover} if additionally
the order of $\|Gal|(k(\Xq)|k(X))$ is prime to $p$. Of course,
this is no restriction for $p=1$.

In the following, algebraic groups will be linear and smooth while
subgroups may also be group schemes. In particular, when $X$ is a
$G$-variety then the stabilizer $G_x$ of a point $x\in X$ will be
considered to be a group scheme. This entails that the orbit $Gx$ is
isomorphic to the quotient $G/G_x$. This also means that the group
$\|Aut|^G(Gx)$ is a group scheme which is isomorphic to $N_G(G_x)/G_x$
where $N_G(H)$ denotes the normalizer subgroup scheme of a subgroup
scheme $H\subseteq G$.

For any group scheme $H$, let $H^\circ$ be its
connected component of unity. Then $\pi_0(H):=H/H^\circ$ is an honest
finite group.

We start with a statement on covers of homogeneous varieties due to
Brion and Szamuely \cite{BrionSzamuely}. For the convenience of the
reader, we provide a slightly different proof.

\begin{proposition}\label{lemma:homocover}
  Let $G$ be a connected linear algebraic group, let $X$ be a
  homogeneous space for $G$, and let $q:\Xq\to X$ be a $p'$-cover. Let
  $\Gq$ be a connected component of $G\times_X\Xq$. Then $\Gq$ carries
  the structure of a connected linear algebraic group such that the
  projection $\qq:\Gq\to G$ is a homomorphism, $\Xq$ is a homogeneous
  $\Gq$-variety, and $q$ is $\Gq$-equivariant.
\end{proposition}

\begin{proof}
  Since the Galois group of $\qq$ is a subgroup of the Galois group of
  $q$ it is also a $p'$-cover. Thus, by a theorem of Miyanishi
  \cite{Miyanishi}*{Thm.\ 2}, $\Gq$ has the structure of a linear
  algebraic group such that $\qq:\Gq\to G$ is a group homomorphism
  with finite kernel. We need to show that $\Xq$ has the structure of
  a homogeneous $\Gq$-variety such that both $\Gq\to\Xq$ and
  $\Xq\to X$ are $\Gq$-equivariant.

  To this end, choose a point $x_0\in X$ identifing $X$ with $G/H$
  where $H=G_{x_0}$ (possibly a subgroup scheme) and let
  $H^\circ\subseteq H$ be its connected component of unity. Let
  $e\in G$ and $\eq\in\Gq$ be the identity elements. Then
  $\eq=(e,\xq_0)$ where $\xq_0$ is a point of $\Xq$ lying over
  $x_0$. The problem is to show that the preimage $\Hq$ of $\xq_0$ in
  $\Gq$ is a subgroup and that $\Gq\to\Xq$ is the quotient by $\Hq$.

  We first show that $\Hq^\circ$ is a subgroup of $\Gq$. Since the
  orbit map $G\to G/H$ is right $H$-invariant, we get a right
  $H$-action on $G\times_X\Xq$ by $(g,\xq)h=(gh,\xq)$ and therefore a
  right $H^\circ$-action on $\Gq$ such that $\qq$ is
  $H^\circ$-equivariant (recall that $\Gq$ is just a connected
  component of $G\times_X\Xq$). Let $H^1=\eq H^\circ\subseteq\Gq$ be
  the $H^\circ$-orbit of $\eq$. The freeness of the $H$-action on $G$
  implies that $\qq:H^1\to H^\circ$ is an isomorphism of right
  $H^\circ$-varieties. Since $\qq^{-1}(H^\circ)\to H^\circ$ is étale
  it follows that $H^1=\qq^{-1}(H^\circ)^\circ$. The latter being a
  subgroup of $\Gq$, and thus also $H^1$ is one. It follows that the right
  $H^\circ$-orbits of $\Gq$ are precisely the right
  $H^1$-cosets. These map all to a point in $X$ and therefore, by
  connectedness, to a point in $\Xq$. Thus, the morphism $\Gq\to\Xq$
  factors through $\Gq/H^1$ and we get the following diagram:
  \[
    \cxymatrix{
      \Gq\ar[r]\ar[d]^\qq&\Gq/H^1\ar[r]\ar[d]^{q^1}&\Xq\ar[d]^q\\
      G\ar[r]&G/H^\circ\ar[r]&G/H}
  \]
  Thereby, all arrows not involving $\Xq$ are $\Gq$-equivariant. We
  see that $\Xq$ is sandwiched between $\Gq/H^1$ and $\Gq/H^2$ with
  $H^2:=\qq^{-1}(H)\subseteq\Gq$. Moreover, $H^1$ is the connected
  component of $H^2$. Thus, $\Gq/H^1\to\Gq/H^2$ is a Galois cover with
  Galois group $H^2/H^1$. Now Galois theory tells us that $\Xq$ is of
  the form $(\Gq/H^1)/\Gamma$ where $\Gamma\subseteq H^2/H^1$ is a
  subgroup. With $\Hq$ the preimage of $\Gamma$ in $H^2$ we get
  $\Xq=\Gq/\Hq$ as claimed.
\end{proof}

\begin{remark}
  The difference between our proof above and that of Brion-Szamuely is
  that we deduce it from Miyanishi's paper \cite{Miyanishi} while
  \cite{BrionSzamuely} use Orogonzo's \cite{Orogonzo}. Thus, our
  approach is all-in-all more elementary but less conceptual.
\end{remark}

\section{Spherical varieties}

Let $G$ be a connected reductive group with Borel subgroup
$B=TU\subseteq G$. The lattice of weights of $G$ is by definition
$\Xi(G):=\|Hom|(T,\G_m)$ where $\G_m$ is the multiplicative group of
$k$.

Recall that a $G$-variety $X$
is \emph{spherical} if $B$ has a dense open orbit in $X$. A closed
subgroup $H\subseteq G$ is \emph{spherical} if the homogeneous space
$X=G/H$ is spherical.

A \emph{weight} of $X$ is a character $\chi\in\Xi(G)$ such that there
exists a non-zero rational function $f_\chi\in k(X)$ which is
$B$-semiinvariant with character $\chi$, i.e., with
\[
  f_\chi(b^{-1}x)=\chi(b)f(x)
\]
for $b\in B$ and $x\in X$. The sphericity of $X$ implies that
$f_\chi$ is unique up to a non-zero scalar. Let $\Xi(X)$ be the set of
weights. It is a finitely generated free abelian group. Its rank is
called the \emph{rank of $G/H$}.

A \emph{color} of $X=G/H$ is by definition a $B$-invariant irreducible
divisor of $X$. Each color $D$ comes with a homomorphism
\[
  \delta_D:\Xi(X)\to\ZZ:\chi\mapsto v_D(f_\chi)=:\<\delta_D,\chi\>.
\]
where $v_D$ is the discrete valuation of $k(X)$ induced by $D$. The
set of colors will be denoted by $\cD(X)$.

Let $\Xi_\QQ(X):=\Xi(X)\otimes\QQ$ be the set of fractional weights of
$X$. Then $\delta_D$ extends uniquely to a $\QQ$-linear map
$\delta_D:\Xi_\QQ(X)\to\QQ$ and we define, following
\cite{Hofscheier}*{Def.\ 1.6} the \emph{$\cD$-saturation of $\Xi(X)$}
as
\[
  \Xi^\circ(X):=\{\chi\in\Xi_\QQ(X)\mid\<\delta_D,\chi\>\in\ZZ\text{ for all }D\in\cD(X)\}.
\]
It contains the subgroup
\[
  \Xi^\circ_G(X):=\Xi^\circ(X)\cap\Xi(G)
\]
which itself contains $\Xi(X)$ as a subgroup of finite index. The
saturations are isogeny invariants:

\begin{lemma}\label{lemma:Xi0isogeny}
  Let $\gamma:\Xq\to X$ be an equivariant étale morphism of
  homogeneous spherical $G$-varieties. Then $\Xi^\circ(\Xq)=\Xi^\circ(X)$.
\end{lemma}

\begin{proof}
  Since $\Xi(X)\subseteq\Xi(\Xq)$ and both groups have the same rank
  we get $\Xi_\QQ(\Xq)=\Xi_\QQ(X)$. Since $\gamma$ is unramified along each
  color $\Dq\subset\Xq$, the restriction of $v_\Dq$ to $k(X)$ equals
  $v_D$ with $D=\gamma(\Dq)$. Hence, $\delta_\Dq=\delta_D$ on
  $\Xi_\QQ(X)$. Since $\cD(\Xq)\to\cD(X)$ is surjective, the assertion
  follows.
\end{proof}

Let $S^\vee\subset\|Hom|(\Xi(G),\ZZ)$ be the set of
simple coroots of $G$ and let
\[\label{eq:XiG}
  \Xi^\circ(G):=\{\chi\in\Xi_\QQ(G)\mid\<S^\vee,\chi\>\in\ZZ\}.
\]
This group can also be characterized as the union of all weight
lattices $\Xi(\Gq)$ where $\Gq$ runs through all isogenies $\Gq\to G$.

\begin{remark}
  This is actually consistent with the notation $\Xi^\circ(X)$ if one
  considers $X:=G$ as a spherical $G\times G$-variety. Then $\Xi(G)$
  can be identified with $\Xi(X)$ and $S^\vee$ with the set of colors
  of $X$. We won't use this in the following.
\end{remark}

For the computation of the fundamental group later on, we need the
following fact:

\begin{lemma}\label{lemma:LunaKnop}
  Let $X$ be a homogeneous spherical variety. Then
  $\Xi^\circ(X)\subseteq\Xi^\circ(G)$.
\end{lemma}

\begin{proof}
  By Luna \cite{LunaGC}*{Prop.\ 3.4} (for $p=1$) and
  \cite{KnopLocalization}*{Prop.\ 2.3} (for $p\ge1$), the restriction
  of every simple coroot to $\Xi_\QQ(X)$ is a $\ZZ$-linear combination
  of the forms $\delta_D$ where $D$ is a color of $X$. The assertion
  follows.
\end{proof}

The following lemma is the technical core of our approach. It asserts
the existence of enough coverings.

\begin{lemma}\label{lemma:CoverExist}
  Let $X=G/H$ be spherical and let $\Gamma\subseteq\Xi^\circ(X)$ be a
  subgroup with $\Xi(X)\subseteq\Gamma$ and $[\Gamma:\Xi(X)]$ finite
  and coprime to $p$. Then there is a cover $\Xq\to X$ such that
  $\Gamma=\Xi(\Xq)$ (considered as a subgroup of
  $\Xi_\QQ(\Gq)=\Xi_\QQ(G)$) where $\Gq$ is as in
  \cref{lemma:homocover}. Moreover, $\Gq=G$ if
  $\Gamma\subseteq\Xi(G)$.
\end{lemma}

\begin{proof}
  Because of \cref{lemma:Xi0isogeny} we may assume by induction that
  $\Gamma/\Xi(X)$ is cyclic of order $d\ge0$. Let
  $\chi\in\Xi^\circ(X)$ be a lift of the generator. Let $\Tq$ be the
  torus with $\Xi(\Tq)=\Gamma+\Xi(G)=\ZZ\chi+\Xi(G)$. Then the
  inclusion of $\Xi(G)$ in $\Xi(\Tq)$ induces an isogeny $\Tq\to
  T$. Since $B=T\ltimes U$, this isogeny extends to an isogeny
  $\Bq:=\Tq\ltimes U\to B$ by means of which $\Bq$ acts on $X$ with
  the same dense orbit $Bx_0$. This in turn induces an inclusion
  $k(X)=k(Bx_0)\into k(\Bq)$ of function fields. Thereby, a
  semiinvariant $f_\eta$ on $X$ is mapped to the character $\eta$ of
  $\Bq$.

  Consider the character $\chi$ as an element of $k(\Bq)$ and denote
  it as such by $f$. Let $F:=k(X)(f)$ be the subfield generated by $f$
  over $k(X)$. Since by construction $d\chi\in\Xi(X)$, we have
  $f^d=f_{d\chi}\in k(X)$ which shows that $F$ is a cyclic Galois
  extension of $k(X)$ of degree $d$. Observe also that $F$ is stable
  for the action of $\Bq$.
  
  Let $\gamma:\Xq\to X$ be the integral closure of $X$ in $F$. This
  means that $\Xq$ is a normal variety with $k(\Xq)=F$ and $\gamma$ is
  a finite $B$-morphism. Since $\Bq$ acts on $X$ and $f$, it also acts
  on $\Xq$ such that $\gamma$ is equivariant.

  We claim that $\Xi(\Xq)=\Gamma$. Indeed,
  $\Gamma=\Xi(X)+\ZZ\chi\subseteq\Xi(\Xq)$ and
  $\Xi(B)+\ZZ\chi=\Xi(\Bq)$ by construction. Thus, if
  $\eta\in\Xi(\Xq)$ then $\eta=\eta_0+a\chi$ with
  $\eta_0\in\Xi(B)$ and
  $a\in\ZZ$. Hence
  $\eta_0=\eta-a\chi\in\Xi(B)\cap\Xi(\Xq)$. But the extensions
  $k(\Xq)|k(X)$ and
  $k(\Bq)|k(B)$ have the same cyclic Galois group $\pi$ of order
  $d$. Thus $\eta_0\in\Xi(\Xq)^\pi=\Xi(X)$ which implies the claim.

  Next we claim that $\gamma$ is globally unramified, i.e., an étale
  cover. Then the first assertion of the lemma would follow
  immediately from \cref {lemma:homocover}.

  To show the claim recall that the irreducible components
  of $X\setminus Bx_0$ are the colors of $X$. Let $D$ be one of them and let
  $E\subset\Xq$ be a prime divisor lying above $D$. Then $\cO_{X,D}$
  is a discrete valuation ring and $\cO_{\Xq,E}$ is its integral
  closure in $F$. Now we use the assumption $\chi\in\Xi^\circ(X)$
  which means $\<\delta_D,\chi\>\in\ZZ$. This implies
  $v_D(f_{d\chi})\in d\ZZ$ and therefore $f_{d\chi}=ug^d$ where $g$ is
  a uniformizer for $D$ and $u\in\cO_{X,D}^\times$. Hence $F|k(X)$ is
  also generated by the $d$-th root of the unit $u$ implying that $F$
  is unramified in $E$ over $k(X)$.  Since this holds for any color we
  see that $\gamma$ is unramified in codimension one. From Zariski's
  theorem on the purity of the branch locus (see, e.g., \cite{Altman})
  we infer that $\gamma$ is globally unramified proving the claim.

  Now assume $\chi\in\Xi(G)$. Then we have $\Bq=B$, i.e., the
  $B$-action on $X$ lifts to a $B$-action on $\Xq$. Thus, the
  inclusion morphism $B\into G$ lifts to a homomorphism
  $B\to\Gq\subseteq G\times_X\Xq:b\mapsto(b,b\xq_0)$. Hence, $\Gq\to G$ induces
  an isomorphism of Borel subgroups which implies
  $\Gq\overset\sim\to G$, e.g., by the isogeny theorem.
\end{proof}

\begin{corollary}\label{cor:conChar0-p}
  Let $X=G/H$ be spherical and $X^\circ=G/H^\circ$. Then
  $\Xi(X^\circ)$ is a subgroup of $\Xi^\circ_G(X)$. The index is a
  power of $p$.
\end{corollary}

  \begin{proof} The first assertion follows
  from \cref{lemma:Xi0isogeny}:
  \[
    \Xi(X^\circ)\subseteq
    \Xi^\circ(X^\circ)\cap\Xi(G)=\Xi^\circ(X)\cap\Xi(G)=\Xi^\circ_G(X).
  \]
  For the second assertion we apply \cref{lemma:CoverExist} to the
  preimage $\Gamma\subseteq\Xi^\circ_G(X)$ of the $p'$-part of
  $\Xi^\circ_G(X)/\Xi(X)$. This yields $\Gq=G$ and a subgroup
  $\Hq\subseteq H$ of finite index with
  \[
    \Gamma=\Xi(G/\Hq)\subseteq\Xi(G/\Hq^\circ)=\Xi(X^\circ).\qedhere\hfill\qed
  \]
\end{proof}

In the remainder of the paper keep in mind that $k^*$ is a divisible
abelian group without $p$-torsion. Thus, the functor
$\|Hom|(\text{--},k^*)$ is exact and kills (exactly) the $p$-torsion.

\begin{lemma}\label{lemma:autom}
  Let $E$ be an algebraic group of $G$-automorphisms of the spherical
  variety $X=G/H$. Then the map
  \[\label{eq:Phi}
    \Phi:E\to\|Hom|(\Xi(X),k^*):\phi\mapsto
    [\chi\mapsto\frac{{}^\phi f_\chi}{f_\chi}]
  \]
  is a homomorphism. Its kernel is a unipotent group and its image is
  $\|Hom|(\Xi(X)/\Xi(X/E),k^*)$.
\end{lemma}

\begin{proof}
  Observe that $\frac{{}^\phi f_\chi}{f_\chi}$ is a non-zero
  $B$-invariant function on $X$, hence can be considered as an element
  of $k^*$. For a similar reason, $\Phi(\phi)(\chi)$ is additive in
  $\chi$ and multiplicative in $\phi$. So $\Phi$ is a well-defined
  homomorphism which is independent of the choice of the $f_\chi$.

  Now let $Bx_0\subseteq X$ be the open $B$-orbit and $B_0=B_{x_0}$
  the stabilizer of $x_0$ (a subgroup scheme of $B$). By uniqueness,
  the open $B$-orbit is fixed by $E$. Hence we can consider $E$ as a
  subgroup of $\|Aut|^BBx_0=N_B(B_0)/B_0$. Then $A:=B/B_0U$ is a torus
  with $\Xi(A)=\Xi(X)$, hence $\|Hom|(\Xi(X),k^*)=A(k)$. Now consider
  the composition $E\overset\sim\to Ex_0\subseteq Bx_0\to A$. Then on
  a level of $k$-points this homomorphism is exactly $\Phi$. The
  kernel $Ex_0\cap Ux_0$ is unipotent since it is contained in
  $\|Aut|^U(Ux_0)=N_U(B_0\cap U)/(B_0\cap U)$. Finally, let
  $E_0:=\Phi(E)\subseteq A$. Since $\Xi(X/E)=\Xi(A/E_0)$
  \[
    \begin{split}
    \|Hom|(\Xi(X)/\Xi(X/E),k^*)=&\|Hom|(\Xi(A)/\Xi(A/E_0),k^*)=\\
      =&\|ker|(\|Hom|(\Xi(A),k^*)\to\|Hom|(\Xi(A/E_0),k^*))=\\
      =&\|ker|(A\to A/E_0)=E_0.\hspace{150pt}\qedhere
    \end{split}
  \]
\end{proof}

\begin{remark}
  In the proof, we just used the existence of an open
  $B$-orbit. Considering also the fact that the $B$-action is coming
  from a $G$-action one can show (see, e.g., \cite{KnopLV}*{Thm.\
    6.1}) that the kernel of $\Phi$ is finite,
  hence a $p$-group.
\end{remark}

We are coming to the answer of Leslie's question. In characteristic
zero, it is due to Hofscheier \cite{Hofscheier}*{Cor.~1.5}.

\begin{theorem}\label{cor:H/Hcirc}
  Let $X=G/H$ be spherical. Then there is a
  canonical surjective homomorphism
  \[
    \pi_0(H)=H/H^\circ\to\|Hom|(\,\Xi^\circ_G(X)/\Xi(X),k^*)
  \]
  whose kernel is a $p$-group. In particular, $\pi_0(H)$ is a
  $p$-group if and only if $\Xi^\circ_G(X)/\Xi(X)$ is one.
\end{theorem}

\begin{proof}
  Let $X^\circ=G/H^\circ$. Then an application of \cref{lemma:autom} to
    $(X,E)=(X^\circ,H/H^\circ)$ yields a surjective homomorphism
  \[
    H/H^\circ\to\|Hom|(\,\Xi(X^\circ)/\Xi(X),k^*)
  \]
  whose kernel is a $p$-group. Since $k^*$ is divisible and has no
  $p$-torsion, the assertion follows now from \cref{cor:conChar0-p}.
\end{proof}

\begin{remarks}
  \emph{a)} Let $X=G_0$ be a semisimple group considered as a
  homogeneous space for $G=G_0\times G_0$ and let
  $H=(Z(G_0)\times\{e\})\Delta(G_0)$. Then $H^\circ=\Delta(G_0)$ and
  therefore $\pi_0(H)=Z(G_0)$. Thus, $\pi_0(H)$ can be an arbitrary
  finite abelian $p'$-group.

  \emph{b)} More generally, assume $H^\circ$ is normalized by a torus $A_0$
  with $E_0:=A_0\cap H^\circ$ finite. Then for any finite subgroup
  $E\subseteq A/E_0$ one has $\pi_0(EH^\circ)=E$. An example would be
  $H^\circ=U$ and $A=T$.

  \emph{c)} In characteristic zero, one can show that these are the
  only effects which make $\pi_0$ big, namely it follows from Losev
  \cite{Losev} that if $G$ is of adjoint type and $H^\circ$ is of
  finite index in its normalizer then $\pi_0(H)$ divides
  $2^{\|rk|G}$.

  \emph{d)} The simplest example for $[H:H^\circ]$ being divisible by
  $p>1$ is $G=SL(2)$ and $H$ the normalizer of the maximal torus
  $T\cong\G_m$. Then $G/H$ is spherical with $[H:H^\circ]=2$ for any
  $p$. Thus for $p=2$, the index $[H:H^\circ]$ is indeed divisible by
  $p$.

\end{remarks}

Finally, if one replaces $G$ by a covering group $\Gq$ and takes the
limit, one obtains:

\begin{theorem}
  Let $X=G/H$ be spherical. Then there is an isomorphism
  \[
    \pi_1(X)_{p'}\overset\sim\to\|Hom|(\Xi^\circ(X)/\Xi(X),k^*).
  \]
\end{theorem}

  \begin{proof}
    For a fixed cover $\Gq\to G$ let $\pi_1(X;\Gq)$ be the quotient of
    $\pi_1(X)$ which classifies all covers $\Xq\to X$ such that the
    $G$-action on $X$ lifts to a $\Gq$-action on $\Xq$. Let
    $\Hq\subseteq\Gq$ be the preimage of $H$ such that
    $X=\Gq/\Hq$. According to \cref{cor:H/Hcirc} we have
  \[
    \pi_1(X;\Gq)_{p'}=\pi_1(\Gq/\Hq;\Gq)_{p'}
    =(\Hq/\Hq^\circ)_{p'}=
    \|Hom|(\,\Xi^\circ_\Gq(X)/\Xi(X),k^*).
  \]
  On the other hand,
  \cref{lemma:homocover} implies
  \[\label{eq:pi1XG}
    \pi_1(X)_{p'}=\varprojlim_\Gq\pi_1(X;\Gq)_{p'}=
    \varprojlim_\Gq\|Hom|(\,\Xi^\circ_\Gq(X)/\Xi(X),k^*).
  \]
  Let $\Xi^{\circ\circ}(G)\subseteq\Xi_\QQ(G)$ be the union of all
  weight lattices $\Xi(\Gq)$ where $\Gq$ runs over all
  \emph{separable} isogenies $\Gq\to G$. Let moreover
  $\Xi^{\circ\circ}(X):=\Xi^\circ(X)\cap\Xi^{\circ\circ}(G)=\bigcup_{\Gq}\Xi_\Gq^\circ(X)$.
  Since $k^*$ is an injective $\ZZ$-module, \eqref{eq:pi1XG} implies
  \[
    \pi_1(X)_{p'}=\|Hom|(\Xi^{\circ\circ}(X)/\Xi(X),k^*).
  \]
  Recall that the group $\Xi^\circ(G)$ defined in \eqref{eq:XiG} is
  the union of all $\Xi(\Gq)$ where $\Gq$ runs over all, possibly
  inseparable, isogenies $\Gq\to G$. Since every such isogeny is the
  composition of a purely inseparable isogeny followed by a separable
  one, we see that the quotient $\Xi^\circ(G)/\Xi^{\circ\circ}(G)$ is
  a $p$-group.  It follows from \cref{lemma:LunaKnop} that
  \[
    \frac{\Xi^\circ(X)}{\Xi^{\circ\circ}(X)}
    \subseteq\frac{\Xi^\circ(G)}{\Xi^{\circ\circ}(G)}.
  \]
  Since $k^*$ is injective without $p$-torsion, this implies
  \[
    \pi_1(X)_{p'}=\|Hom|(\Xi^{\circ\circ}(X)/\Xi(X),k^*)=
    \|Hom|(\Xi^\circ(X)/\Xi(X),k^*).\qedhere\hfill\qed
  \]
\end{proof}

\begin{remark}
  There is another approach to compute the fundamental group of a
  homogeneous space $X=G/H$ due to Demarche \cite{Demarche} which
  presupposes the knowledge of $H$. The point of the present note is to
  use just (part of) the spherical system of $X$.
\end{remark}

\end{document}